\documentclass [12pt]{article}
\usepackage{amsmath}
\usepackage{amssymb}
\usepackage{cite}
\usepackage{hyperref}
\newsymbol \blackbox 1004
\newcommand{\eh}{\hfill}\newlength{\sperr}

\newenvironment{proof}{{\settowidth{\sperr}{\bf\rm
Proof}%
\par\addvspace{0.3cm}\noindent\parbox[t]{1.3\sperr}
{\bf\rm P\eh r\eh o\eh o\eh f\eh }%
}}{\nopagebreak\mbox{}
$\blackbox$\par\addvspace{0.3cm}}
\numberwithin{equation}{section}

\newtheorem{theorem}{Theorem}

\newtheorem{corollary}[theorem]{Corollary}

\newtheorem{example}[theorem]{Example}

\newtheorem{proposition}[theorem]{Proposition}
\newtheorem{remark}[theorem]{Remark}

\def\di{\mathrm{diag}}

\def\a{\alpha}
\def\b{\beta}

\def\S{\Sigma}

\def\Lam{\Lambda}
\def\om{\omega}
\def\Om{\Omega}
\def\de{\delta}

\def\vp{\varphi}

\def\vt{\vartheta}
\def\wh{\widehat}
\def\wt{\widetilde}
\def\ov{\overline}
\def\p{\partial}

\def\BC{{\mathbb C}}

\def\cla{{\mathcal A}}

\newcommand{\cV}{{\mathcal V}}

\newcommand{\wti}{\widetilde}

\title{KdV equation in the 
quarter--plane:\\
evolution of the Weyl
functions
and unbounded solutions}

\author{A. Sakhnovich \footnote{ E-mail:   Oleksandr.Sakhnovych@univie.ac.at}}

\date{}

\begin{document}

\maketitle

\noindent {\bf Abstract.} The matrix KdV equation with a negative dispersion term
is considered in the right upper quarter--plane. The evolution law is derived for
the Weyl function of a corresponding auxiliary linear system. 
Using the low energy asymptotics of the Weyl functions, the unboundedness
of solutions is obtained for some classes of  the initial--boundary conditions. 
\vspace*{0.5cm}



\vspace*{1cm}

\setcounter{equation}{0}
\section{Introduction} \label{s1}
We consider the matrix KdV equation with the minus sign in front of the dispersion term $u_{xxx}$:
\begin{equation} \label{1.1}
u_t+3(uu_x+u_xu)-u_{xxx}=0,
\end{equation}
where $u(x,t)$ is an $m \times m$ matrix function. Equation (\ref{1.1})  is the compatibility condition
of the auxiliary linear systems
\begin{align}
\Phi_x(x,t, z) &=G(x,t, z)\Phi(x,t, z), \label{1.2} \\
\Phi_t(x,t, z) &=F(x,t, z)\Phi(x,t, z), \label{1.2a}
\end{align}
\begin{eqnarray} \label{1.3}
&&G:= \left[
\begin{array}{cc}
0 & I_m \\ u- z I_m & 0
\end{array}
\right], \\
\label{1.4}
&& F: = \left[
\begin{array}{cc}
u_x & -2(u+2 z I_m)\\u_{xx}- 2 (u+2 z I_m)(u- z I_m)& -u_x
\end{array}
\right],
\end{eqnarray}
where
$I_m$ is the $m \times m$ identity matrix.
 In other words equation (\ref{1.1}) is equivalent to the zero curvature equation
\begin{equation} \label{1.4a}
G_t-F_x+[G,F]=0, \quad [G,F]:=GF-FG,
\end{equation}
where $G$ and $F$ are given by  (\ref{1.3}) and (\ref{1.4}), respectively.

Initial-boundary value problems for the integrable nonlinear equations
(and KdV equation, in particular) are of great interest, see, for instance, 
\cite{BW83, BoFo, FoL, SaA7, Sa88, Ton} and references therein.
System \eqref{1.2}, \eqref{1.3} is equivalent 
to the canonical system \eqref{1.13} (and to the Schr\"odinger equation),
and in this paper we derive
the evolution $M(t,z)$ of the Weyl function of this system. This evolution
is an important component of the solution of the initial-boundary value
problem. For simplicity, we derive the evolution  under condition that $F$ and $G$
are continuously differentiable, though the requirement 
of the continuous differentiability could be weakened using the results from \cite{SaAF}. 

If $u(0,t)=u_{xx}(0,t)=0$, then system
\eqref{1.2a} at $x=0$ is equivalent to a Dirac system and its Weyl function
is expressed via $M(0,z)$ (see formula \eqref{2.21}).  We apply \eqref{2.21}
and low energy asymptotics of  $M(0,z)$ to show 
the unboundedness
of the KdV solutions in the quarter--plane for some classes of simple
initial conditions $u(x,0)$.

Our Weyl function $M(t,z)$ is connected with the Weyl function from \cite{CG0}
(the latter being denoted here by ${\cal{M}}(t,z)$) via the linear fractional transformation $M=({\cal{M}}-I_m)({\cal{M}}+I_m)^{-1}$.
We note that the high energy asymptotics of the Weyl functions
was actively studied (see \cite{CG0, CG1, KoST2, SaA02} and references therein)
following the seminal papers \cite{GS, GeSi}. Though the low energy asymptotics of the Weyl functions is used in the present paper, the high energy asymptotics
(namely, an important result on asymptotics of the Weyl function in terms of
the values of $u$ and its derivatives at $x=0$  from  \cite{CG0})
jointly with the evolution of the Weyl function could also prove useful for 
the analysis of the initial-boundary conditions.

We discuss some background in Section \ref{sb},
obtain the evolution law in Section \ref{kdv}, and study the
unboundedness of the solutions in Section \ref{non}

\vspace*{0.5cm}
\setcounter{equation}{0}
\section{Some Background} \label{sb}

Let us  normalize the fundamental solution $\Psi$ of the equation \eqref{1.2}
by introducing
\begin{equation}
\Psi(x,t, z) = \Phi(x,t, z) \Phi(0,t, z)^{-1} \label{1.4b}
\end{equation}
satisfying the initial condition
\begin{equation} \label{1.5}
\Psi(0,t, z)=I_{2m}.
\end{equation}

Suppose, $G$ and $F$ are continuously differentiable on the half--strip $0 \leq x<\infty$,
$0 \leq t < {\mathbf t}\leq \infty$ and (\ref{1.4a}) holds. Then, according to section 12.1 \cite{Sa99} (see also \cite{Sa87, Sa88}) we have
\begin{equation} \label{1.6}
\Psi(x,t, z)=V(x,t, z)\Psi(x,0, z)V(0,t, z)^{-1},
\end{equation}
where the $2m \times 2m$ matrix function $V$ satisfies relations
\begin{equation} \label{1.7}
V_t(x,t, z)=F(x,t, z)V(x,t, z), \quad V(x,0, z)=I_{2m}.
\end{equation}
Introduce the matrices
\begin{equation} \label{1.8}
J:= \left[
\begin{array}{cc}
0 & I_m \\ I_m & 0
\end{array}
\right], \quad \Sigma_3: =\left[
\begin{array}{cc}
I_m & 0\\ 0 & -I_m
\end{array}
\right],
\end{equation}
\begin{equation} \label{1.9}
J_1=TJT^*=i \left[
\begin{array}{cc}
0 & -I_m \\ I_m & 0
\end{array}
\right], \quad
T:=\frac{1}{\sqrt{2}}
\left[
\begin{array}{cc}
i I_m & I_m \\ i I_m & -I_m
\end{array}
\right].
\end{equation}
Further we shall consider the case of the self-adjoint (real-valued for $m=1$) $u$:
\begin{equation} \label{1.10}
u(x,t)=u(x,t)^*, \, {\mathrm{i.e.}}, \, \frac{\p}{\p x} \Big(\Psi(x,t,0)^*J_1\Psi(x,t,0)\Big)=0.
\end{equation}
From (\ref{1.5}), (\ref{1.9}) and (\ref{1.10}) it follows that
\begin{equation} \label{1.11}
T^*\Psi(x,t,0)^*J_1\Psi(x,t,0)T=T^*J_1T=J.
\end{equation}
Putting
\begin{equation} \label{1.12}
\wt \Psi(x,t, z)=\big(\Psi(x,t,0)T\big)^{-1}\Psi(x,t, z)T,
\end{equation}
and taking into account (\ref{1.2}), (\ref{1.3}), (\ref{1.11}) and (\ref{1.12})
we see that $\wt \Psi(x,t, z)$ is the fundamental solution
of the canonical system
\begin{equation} \label{1.13}
\wt \Psi_x(x,t, z)=i z J H(x,t)\wt \Psi(x,t, z), \quad \wt \Psi(0,t, z)=I_{2m},
\end{equation}
where
\begin{equation} \label{1.14}
H(x,t)=T^*\Psi(x,t,0)^* \left[
\begin{array}{cc}
I_m & 0 \\ 0 & 0
\end{array}
\right] \Psi(x,t,0)T \geq 0.
\end{equation}
Moreover, $H$ satisfies \cite{Sa87} the positivity condition
\begin{equation} \label{1.15}
\int_0^lH(s,t) ds>0 \quad (l>0).
\end{equation}
Indeed, for any $h \in \BC^{2m}$, $h\not=0$ we have,
\begin{equation} \label{1.15.1}
h^*H(s,t)h=g(s,t)^*g(s,t), \quad g(s,t):=[I_m \quad 0] \Psi(s,t,0) Th,
\end{equation}
where, according to (\ref{1.2}), (\ref{1.3}), and \eqref{1.5}, the relations
\begin{equation} \label{1.15.2}
g_{ss}(s,t)=u(s,t)g(s,t), \quad \left[ \begin{array}{c}
g(0,t) \\ g_s(0,t)
\end{array}
\right] =Th \not= 0
\end{equation}
hold. Inequality
(\ref{1.15}) follows from (\ref{1.15.1}) and (\ref{1.15.2}) .

By \eqref{1.15}, the linear fractional transformations
\begin{align} \label{1.16}
\begin{split}
M( l, t, z) &= i\Big(\cla_{11}(l, t, z)P_l( t, z)+\cla_{12}(l,t, z)Q_l(t, z)\Big) \\
& \quad \times \Big(\cla_{21}(l,t, z)P_l(t, z)+\cla_{22}(l,t, z)Q_l(t, z)\Big)^{-1},
\quad \Im (z) >0,
\end{split}
\end{align}
where the matrices $\cla_{kj}$ are the $m \times m$ blocks of $\cla$,
\begin{equation} \label{1.17}
\cla(l,t, z):=\wt \Psi(l,t,\ov z)^*,
\end{equation}
and $P_l$, $Q_l$ are meromorphic nonsingular pairs with property-$J$,
\begin{equation} \label{1.18}
P_l^*P_l+Q_l^*Q_l>0, \quad P_l^*Q_l+Q_l^*P_l \geq 0,
\end{equation}
are well-defined for $\Im(z)>0$.
The matrix functions $M$ are Herglotz (Nevanlinna) functions, that is, $\Im(M( z)) \geq 0$
in $\BC_+$, and they are called Weyl--functions of the canonical system on
the interval $(0, \, l)$. Further we shall assume that $u$ is bounded:
\begin{equation} \label{2.2'}
\sup_{0 \leq x<\infty, \, 0 \leq t < {\mathbf t}} \| u(x,t) \|<C.
\end{equation}
Then, by (\ref{1.13}) and (\ref{1.15}) there is a unique limit of the functions
$M( l, t, z)$, which is independent of the choice of the pairs $P_l$, $Q_l$
with property-$J$:
\begin{equation} \label{1.19}
\lim_{l \to \infty} M( l, t, z)=M(t, z).
\end{equation}
Fore a detailed proof of \eqref{1.19} see p. 177 in  \cite{Sa99}, where  the proof  of 
a similar formula (1.18) (condition b)) from p. 169 is given.

Note that one can omit the variable $t$ in formulas   \eqref{1.2},  \eqref{1.4b},
 \eqref{1.5},  \eqref{1.12}--\eqref{1.19} while considering
 a certain subclass of canonical systems.
The limit $M(z)=\lim_{l \to \infty} M( l, z)$ is called the Weyl--function of the system (\ref{1.13})
on the semi-axis $x>0$. It has the property (see formula (1.24) on p. 121 in \cite{Sa99})
\begin{equation} \label{r3}
\int_0^\infty
\left[ \begin{array}{lr}
I_m &i M (z)^* \end{array} \right]
\wt  \Psi(x, z)^*H(x)\wt \Psi(x, z)
\left[ \begin{array}{c}
I_m \\ -i M (z) \end{array} \right]
dx            < \infty, \quad z \in \BC_+.
\end{equation}
The function $M(z)$ is also the Weyl--function of the Sturm--Liouville
system 
\begin{equation} \label{r1}
-Y_{xx}(x,z)+u(x)Y(x,z)=zY(x,z),
\end{equation}
where the matrix function $u$ coincides with the $u$ in \eqref{1.3}.
In particular, formula \eqref{r3} can be rewritten in the form
\begin{equation} \label{r4}
\int_0^\infty
\left[ \begin{array}{lr}
I_m &i M (z)^* \end{array} \right]
Y(x,z)^*Y(x,z)
\left[ \begin{array}{c}
I_m \\ -i M (z) \end{array} \right]
dx            < \infty, \quad z \in \BC_+,
\end{equation}
where $Y$ is the $m \times 2m$ solution of  \eqref{r4} normalized
by the condition 
\begin{equation} \label{r5}
Y(0,z)=(\sqrt{2})^{-1}[iI_m \quad I_m], \quad Y_x(0,z)=(\sqrt{2})^{-1}[iI_m \quad -I_m].
\end{equation}

We also recall that the Weyl--function $M_D(\zeta)$ of the Dirac--type system on the semi-axis
\begin{equation} \label{1.20}
\frac{d}{dt}W(t,\zeta )=i[\zeta \Sigma_3 + \Sigma_3\cV(t)]W(t,\zeta ), \quad W(0,\zeta)=I_{2m}, \quad \cV=\left[\begin{array}{cc}
0&v\\v^{*}&0\end{array}\right],
\end{equation}
where $\cV$ is locally summable, is uniquely defined by the inequality
\begin{equation} \label{1.21}
\int_0^{\infty}[I_m \quad iM_D(\zeta)^*]K W(t, \zeta)^*  W(t, \zeta)K^* \left[
\begin{array}{c}
I_m\\ -iM_D(\zeta)
\end{array}
\right]dt<\infty,
\end{equation}
\begin{equation} \label{1.22}
\Im (\zeta) >0, \quad K:=\frac{1}{\sqrt{2}}
\left[
\begin{array}{cc}
I_m & -I_m \\ I_m & I_m
\end{array}
\right].
\end{equation}
See the procedure to recover $\cV$ from $M_D$ in \cite{SaA02, Sa99} and the references therein.

Using (\ref{1.6}) and (\ref{1.19}) the evolution of the Weyl--function
$M(t, z)$ $(t>0)$ was derived in \cite{Sa87} -\cite{Sa99} for the
KdV equation $u_t-3(uu_x+u_xu)+u_{xxx}=0$
with the plus sign in front of the dispersion term.
Moreover, the initial--boundary problem $u(x,0)=f(x)$, $u(0,t)=u_{xx}(0,t)=0$
was treated in \cite{Sa87} for the scalar case $u_t-6uu_x+u_{xxx}=0$.
We shall modify these results for the case of the KdV equation (\ref{1.1}),
where this number of the initial--boundary conditions will be appropriate
(see \cite{BW83, Ton}).


\vspace*{0.5cm}
\setcounter{equation}{0}
\section{The KdV Equation with a Negative Dispersion Term} \label{kdv}


Denote the Weyl--function of system (\ref{1.13}) at $t=0$ by $M(0, z)$
and put
\begin{eqnarray} \label{2.1}
&& R(l,t, z):=\big(\Psi(l,t,0)T\big)^*\big(V(l,t, \ov z)^*\big)^{-1}\Big(\big(\Psi(l,0,0)T\big)^*\Big)^{-1},
\\ \label{2.2} &&R(t, z)=\left[ \begin{array}{cc}
r_{11}(t, z) & r_{12}(t, z) \\ r_{21}(t, z) & r_{22}(t, z)
\end{array}
\right] :=R(0,t, z),
\end{eqnarray}
where $r_{kj}$ are $m \times m$ blocks of $R$.


\vspace*{0.25cm}

\begin{proposition} \label{evol}
Let the bounded $m \times m$ matrix function $u$
satisfy the KdV equation (\ref{1.1}) on the half--strip
$0 \leq x<\infty$,
$0 \leq t < {\mathbf t}\leq \infty$.
Assume that
the corresponding matrix functions
$G$ and $F$ given by (\ref{1.3}) and (\ref{1.4}) are continuously
differentiable. Then the evolution of the Weyl--function $M(t, z)$
is given by the formula
\begin{equation} \label{2.3}
M(t, z)=i\Big((-i)r_{11}(t, z)M(0, z)+r_{12}(t, z)\Big)
\Big((-i)r_{21}(t, z)M(0, z)+r_{22}(t, z)\Big)^{-1}.
\end{equation}
\end{proposition}
\begin{proof}.
Taking into account (\ref{1.12}) and (\ref{1.17}), rewrite formula (\ref{1.6})
in the form
\begin{equation} \label{2.4}
\cla(l,t, z)R(l,t, z)=R(t, z)\cla(l,0, z).
\end{equation}
To show that $R$ is $J$-expanding in some domain in $\BC_+$,
we shall use the equation
\begin{equation} \label{2.5}
\frac{\p}{\p t}\big(V(l,t, \ov z)^{-1}\big)=-V(l,t, \ov z)^{-1}F(l,t,\ov z)
\end{equation}
From (\ref{2.5}) it follows that
\begin{eqnarray} \label{2.6}
&&\frac{\p}{\p t}\Big(V(l,t, \ov z)^{-1}J_1\big(V(l,t, \ov z)^{-1}\big)^*\Big)
\\ \nonumber &&
=
-V(l,t, \ov z)^{-1}\big(F(l,t,\ov z)J_1+J_1F(l,t,\ov z)^*\big)\big(V(l,t, \ov z)^{-1}\big)^*.
\end{eqnarray}
By (\ref{1.4}) and the first relation in (\ref{1.10}) we have
\begin{equation} \label{2.7}
F(l,t,\ov z)J_1+J_1F(l,t,\ov z)^*=2i(z -\ov z) \left[
\begin{array}{cc}
2I_m & 0 \\ 0 & 2(z + \ov z)I_m-u(l,t)
\end{array}
\right].
\end{equation}
Taking into account (\ref{2.2'}) and (\ref{2.7}) we derive
\begin{equation} \label{2.8}
-\Big(F(l,t,\ov z)J_1+J_1F(l,t,\ov z)^*\Big)>0 \quad {\mathrm{for}} \quad \Im (z) >0, \, \Re z >C/4.
\end{equation}
In view of (\ref{2.6}), (\ref{2.8}) and the second relation in (\ref{1.7}) we get
\begin{equation} \label{2.9}
V(l,t, \ov z)^{-1}J_1\big(V(l,t, \ov z)^{-1}\big)^*>J_1 \quad {\mathrm{for}} \quad \Im (z) >0, \, \Re z >C/4.
\end{equation}
According to (\ref{1.11}), (\ref{2.1}) and (\ref{2.9}) the inequality
\begin{equation} \label{2.10}
R(l,t, z)^*JR(l,t, z)>J \quad {\mathrm{for}} \quad \Im (z) >0, \, \Re z >C/4
\end{equation}
is true. By (\ref{1.16}), (\ref{1.19}), (\ref{2.4}) and (\ref{2.10}), we derive (\ref{2.3})
for $z$ in the domain $ \Im (z) >0, \, \Re z >C/4$. In view of the analyticity of the Weyl-functions,
it follows that (\ref{2.3}) is valid everywhere in $\BC_+$.
\end{proof}

Consider now the particular case of the initial--boundary value problem in the quarter--plane:
\begin{equation} \label{2.11}
u(x, 0)=f(x), \quad u(0,t)=u_{xx}(0,t)=0 \quad (0 \leq x <\infty, \, 0 \leq t <\infty).
\end{equation}
According to (\ref{1.5}), (\ref{2.1}), (\ref{2.2}) and (\ref{2.5}) we have
\begin{align} \label{2.12}
& R(t, z)=T^*\big(V(0,t, \ov z)^*\big)^{-1}T, \\ \label{2.13}
& \frac{d}{d t}R(t, z)=-T^*F(0,t,\ov z)^*TR(t, z), \quad R(0, z)=I_{2m}.
\end{align}
By (\ref{1.4}) and (\ref{2.11}) one can see that
\begin{equation} \label{2.14}
-F(0,t,\ov z)^*=\left[ \begin{array}{cc}
-u_x(0,t) &- 4 z^2 I_m \\ 4 z I_m & u_x(0,t)
\end{array}
\right].
\end{equation}
Following \cite{Sa87}, let us transform (\ref{2.13}) into the Dirac--type system.
Note for that purpose, that
\begin{eqnarray} \label{2.15}\displaystyle
& T\di \{I_m, \, \sqrt{z}I_m\}\left[ \begin{array}{cc}
0 & - z^2 I_m \\ z I_m & 0
\end{array}
\right] \di \{I_m, \,\frac{1}{ \sqrt{z}}I_m\}T^*=-i z^{\frac{3}{2}}\Sigma_3,& \\
\label{2.16} &
T\di \{I_m, \, \sqrt{z}I_m\} \Sigma_3 \di \{I_m, \,\frac{1}{ \sqrt{z}}I_m\}T^*=J,&
\end{eqnarray}
where $J$ and $j$ are defined in (\ref{1.8}) and diag means a block diagonal matrix.
We consider $ z \in \BC_+$ and choose the branch $\sqrt{z}$ so that $\sqrt{z} \in \BC_+$.
Now, put
\begin{align}
& \wt R(t, \zeta):=Z( z)^{-1}R(t, z)Z( z), \quad
Z( z):=T^*\di \{I_m, \,\frac{1}{ \sqrt{z}}I_m\}T^*, \label{2.17} \\
& \displaystyle\zeta:=-4 z^{\frac{3}{2}}. \label{2.18}
\end{align}
From (\ref{2.13})-(\ref{2.18})
it follows that $\wt R$ satisfies the Dirac--type system
\begin{equation} \label{2.19}
\frac{d}{d t}\wt R(t, \zeta) = [i\zeta \Sigma_3 -\di\{u_x(0,t), \, u_x(0,t)\}J] \wt R(t, \zeta),
\quad \wt R(0, \zeta)=I_{2m}.
\end{equation}
Recall that the Weyl--function $M_D$ of the Dirac--type system is defined via (\ref{1.21}).
Recall also that the Weyl--function $M_{tr}$ of the Sturm-Liouville system with the trivial
potential $u$ (i.e., $u$ equal to zero) equals $[i \sqrt{z}-1]/[i \sqrt{z}+1]I_m$.
Hence we shall require that
\begin{equation} \label{2.20}
\lim_{t\to \infty}M(t, z)= \frac{i \sqrt{z} - 1}{i \sqrt{z} + 1}I_m.
\end{equation}


\vspace*{0.25cm}

\begin{proposition} \label{prtca}
Assume that there exists a solution $u$
of the KdV equation (\ref{1.1})  on the quarter--plane
$0 \leq x<\infty$,
$0 \leq t <  \infty$,  which satisfies also the conditions of Proposition \ref{evol}
and the initial--boundary value conditions (\ref{2.11}). Suppose that (\ref{2.20})
holds. Then $u$ may be uniquely recovered by the following procedure:

First, the Weyl--function of the Dirac--type system (\ref{2.19})
is recovered for sufficiently large values of $\Im (\sqrt{z})$
by the formula
\begin{equation} \label{2.21}
M_D(-4 z^{\frac{3}{2}})=\frac{1}{\sqrt{z}}(I_m+M(0, z))(I_m-M(0, z))^{-1},
\end{equation}
where $z$ belongs to the sector $\frac{2}{3}\pi<\arg (z) <\pi$.
The matrix function $M(0,z)$ in (\ref{2.21}) is the Weyl--function of the canonical system (\ref{1.13}), (\ref{1.14}) at $t=0$,
which is
determined by the initial condition $u(x,0)=f(x)$.

Next, the matrix-function $u_x(0,t)$
is uniquely recovered from $M_D( z)$, after which $R(t, z)$ is given by (\ref{2.13}) and (\ref{2.14}).
The evolution of the Weyl--function $M(t, z)$ is given by (\ref{2.3}) in terms of $R$ and $M(0,z)$.

Finally, $u(x,t)$ is uniquely recovered from $M(t, z)$.
\end{proposition}
\begin{proof}.  
By (\ref{2.19}) we get
\begin{equation} \label{2.22}
\frac{d}{d t}\wt R(t, \zeta)^* \Sigma_3\wt R(t, \zeta)
=i(\zeta - \ov\zeta)\wt R(t, \zeta)^*\wt R(t, \zeta).
\end{equation}
Formula (\ref{2.22}) and the second relation in (\ref{2.19}) imply that
\[
\wt R(t, \zeta)^* \Sigma_3 \wt R(t, \zeta)- \Sigma_3 <-\de \int_0^t\wt R(s, \zeta)^*\wt R(s, \zeta)ds
\quad {\mathrm{for}} \, \Im (\zeta) > \de /2>0,
\]
or, equivalently, we have
\begin{equation} \label{2.23}
\Sigma_3 - \wt R(t, \zeta)^* \Sigma_3 \wt R(t, \zeta)>\de \int_0^t\wt R(s, \zeta)^*
\wt R(s, \zeta)ds \quad {\mathrm{for}} \, \Im (\zeta) > \de /2>0.
\end{equation}
Next, let us show that for sufficiently large values of $\Im (\sqrt{z})$ and $t$ the inequality
\begin{equation} \label{2.24}
[I_m \quad iM_D(\zeta)^*]K\wt R(t, \zeta)^* \Sigma_3 \wt R(t, \zeta)K^* \left[
\begin{array}{c}
I_m\\ -iM_D(\zeta)
\end{array}
\right]\geq 0,
\end{equation}
where $M_D$ is given by (\ref{2.21}), is valid.
First, take into account (\ref{1.22}) and (\ref{2.17}) and note that
\begin{equation} \label{2.25}
Z( z)K^*=-\frac{1}{\sqrt{2z}}\left[ \begin{array}{cc}
iI_m& \sqrt{z}I_m \\ I_m & i\sqrt{z}I_m
\end{array}
\right].
\end{equation}
Using (\ref{2.17}), (\ref{2.21}) and (\ref{2.25}) we write
\begin{equation} \label{2.25'}
\wt R(t, \zeta)K^* \left[
\begin{array}{c}
I_m\\ -iM_D(\zeta)
\end{array}
\right]=\frac{-2}{\sqrt{2z}}Z( z)^{-1}R(t, z)\left[
\begin{array}{c}
-iM(0, z)\\ I_m
\end{array}
\right](I_m-M(0, z))^{-1}.
\end{equation}
According to Proposition \ref{evol} we have
\begin{equation} \label{2.26}
R(t, z)\left[
\begin{array}{c}
-iM(0, z)\\ I_m
\end{array}
\right]=\left[
\begin{array}{c}
-iM(t, z)\\ I_m
\end{array}
\right]\Big((-i)r_{21}(t, z)M(0, z)+r_{22}(t, z)\Big).
\end{equation}
Taking into account that
\begin{equation} \label{2.27}
Z( z)^{-1}=T\di \{I_m, \,{ \sqrt{z}}I_m\}T, \quad T^* \Sigma_3 T=J_1,
\end{equation}
we obtain
\begin{equation} \label{2.28}
\big(Z( z)^{-1}\big)^* \Sigma_3 Z( z)^{-1}=\frac{1}{2} \left[
\begin{array}{cc}
i(\ov {\sqrt{z}}- \sqrt{z})I_m& (\ov {\sqrt{z}}+ \sqrt{z})I_m\\ ( \ov {\sqrt{z}}+ \sqrt{z})I_m & i( \sqrt{z}-\ov {\sqrt{z}})I_m
\end{array}
\right].
\end{equation}
From (\ref{2.25'}), (\ref{2.26}) and (\ref{2.28}) it follows that
\begin{align}
&[I_m \quad iM_D(\zeta)^*]K\wt R(t, \zeta)^* \Sigma_3 \wt R(t, \zeta)K^* \left[
\begin{array}{c}
I_m\\ -iM_D(\zeta)
\end{array}
\right] \nonumber
\\ \nonumber
& \quad =\om(t, z)^* [i M(t, z)^* \quad I_m]\left[
\begin{array}{cc}
i(\ov {\sqrt{z}}- \sqrt{z})I_m& (\ov {\sqrt{z}}+ \sqrt{z})I_m\\ ( \ov {\sqrt{z}}+ \sqrt{z})I_m & i( \sqrt{z}-\ov {\sqrt{z}})I_m
\end{array}
\right]
\\ \nonumber
& \qquad \times \left[
\begin{array}{c}
-iM(t, z)\\ I_m
\end{array}
\right]\om(t, z) \nonumber \\
& \quad \sim  \frac{8 |\sqrt{z}|^2}{|i\sqrt{z} + 1|^2} \om(t, z)^* \om( t, z) >0 \quad (t \to \infty),
\label{2.29}
\end{align}
where
\begin{equation} \label{2.30}
\om(t, z)=\frac{1}{\sqrt{z}}\Big((-i)r_{21}(t, z)M(0, z)+r_{22}(t, z)\Big)(I_m-M(0, z))^{-1}.
\end{equation}
We recall the choice $\frac{2}{3}\pi<\arg (z) <\pi$, that is, $\Im (\zeta) >0$.
By (\ref{2.20}) and (\ref{2.29}) for sufficiently large values of $t$ we get (\ref{2.24}).

Hence, it follows from (\ref{2.23}) and (\ref{2.24}) that the inequality
\begin{equation} \label{2.31}
\int_0^\infty [I_m \quad iM_D(\zeta)^*]K \wt R(s, \zeta)^*\wt R(s, \zeta)K^* \left[
\begin{array}{c}
I_m\\ -iM_D(\zeta)
\end{array}
\right]ds<\infty
\end{equation}
holds. Thus, $M_D$ is, indeed, the Weyl--function of the Dirac system.
The evolution $M(t, z)$ follows from Proposition \ref{evol}.
For the inverse problem for our canonical system, when $u$ is bounded, see
\cite{Sa99}, p. 116 and references.
\end{proof}

We provide a short {\bf SUMMARY} of the scheme employed:

\vspace{5pt}
\fbox{\begin{minipage}{290pt}
\begin{align*}
& f(x)=u(x,0) \xrightarrow{\text{by \eqref{1.2}, \eqref{1.4b}}} \Psi(x,0,0), \; x\geq 0,
\xrightarrow{\text{by \eqref{1.14}}} H(x,0), \; x\geq 0, \\
& \xrightarrow{\text{by \eqref{1.13}}} \wti \Psi(x,0,z), \; x\geq 0,
\xrightarrow{\text{by \eqref{1.16}, \eqref{1.19}}} M (0,z)
\xrightarrow{\text{by \eqref{2.21}}} M_D (\zeta) \\
& M_D(\zeta) \text{ and } \eqref{2.19}\xrightarrow{\text{by solving an IP}}
u_x(0,t), \; t\geq 0,
\xrightarrow{\text{by \eqref{2.13}}} R(t,z), \; t\geq 0, \\
& M(0,z) \text{ and } R(t,z), \; t\geq 0, \xrightarrow{\text{by \eqref{2.3}}} M(t,z), \; t\geq 0,
\\
& \xrightarrow{\text{by solving an IP}} u(x,t), \; x\geq 0, \, t\geq 0 \\
& \xrightarrow{\text{prove that $u$ solves}} KdV(u)=0, \; x\geq 0, \, t\geq 0.
\end{align*}
\end{minipage}}
\vspace{10pt}

Consider the simplest example.

\vspace*{0.25cm}

\begin{example} \label{Ee0}
Put for simplicity $m=1$, i.e., consider a scalar KdV equation.
The simplest case is the case $u(x,0)=f(x)=0$
(see the initial--boundary value conditions \eqref{2.11}).
The Weyl function $M(0,z)$ of the  Sturm--Liouville system with $u \equiv 0$
is given by the formula
\begin{equation} \label{i1}
M(0, z)= \frac{i \sqrt{z} - 1}{i \sqrt{z} + 1}.
\end{equation}
By \eqref{2.21} it follows that the Weyl function $M_D(\zeta)$ of the Dirac system \eqref{2.19} 
is given by the formula
\begin{equation} \label{i2}
M_D\big(\zeta(z)\big)=\frac{1}{\sqrt{z}}\frac{1+\frac{i \sqrt{z} - 1}{i \sqrt{z} + 1}}
{1-\frac{i \sqrt{z} - 1}{i \sqrt{z} + 1}}=i, \quad \zeta(z)=-4z^{\frac{3}{2}}.
\end{equation}
As $M_D \equiv i$ is the Weyl function of the Dirac system \eqref{2.19}   with a trivial potential
$u_x=0$ we get the fundamental solution
\begin{equation} \label{i3}
\wt R(t, \zeta)=\exp(it \zeta \S_3).
\end{equation}
Hence, taking into account  \eqref{2.17}  we derive
\begin{equation} \label{i4}
 R\big(t, \zeta(z)\big)=Z(z) \exp\big(it \zeta(z) \S_3\big)Z(z)^{-1}, \quad Z( z):=T^*\di \{I_m, \,\frac{1}{ \sqrt{z}}I_m\}T^*.
\end{equation}
Using (\ref{i4}) we can obtain $M(t,z)$. First, rewrite \eqref{2.3}  in the form
\begin{eqnarray} \label{i5}
 M(t,z)&=&i[1 \quad 0]R\big(t, \zeta(z)\big)\\
\nonumber
&&\times \left[ \begin{array}{c}
-iM(0,z)  \\ 1
\end{array}
\right]\left([0 \quad 1]R\big(t, \zeta(z)\big)\left[ \begin{array}{c}
-iM(0,z)  \\ 1
\end{array}
\right]\right)^{-1}.
 \end{eqnarray}
 Next, note that according to (\ref{i1}) and the second equality in (\ref{i4})  we have
 \begin{equation} \label{i6}
 Z(z)^{-1}\left[ \begin{array}{c}
-iM(0,z)  \\ 1
\end{array}
\right]=-
 \frac{2 \sqrt{z} }{i \sqrt{z} + 1}
\left[ \begin{array}{c}
1  \\ 0
\end{array}
\right].
\end{equation}
From  (\ref{i4})  and (\ref{i6}) it follows that
 \begin{equation} \label{i7}
 R\big(t, \zeta(z)\big)\left[ \begin{array}{c}
-iM(0,z)  \\ 1
\end{array}
\right]= \frac{e^{it\zeta(z)}}{i \sqrt{z} + 1}\left[ \begin{array}{c}
\sqrt{z}+i  \\ i\sqrt{z}+1
\end{array}
\right]=e^{it\zeta(z)} \left[ \begin{array}{c}
-iM(0,z)  \\ 1
\end{array}
\right].
\end{equation}
By (\ref{i5})  and (\ref{i7}) we have $M(t,z) \equiv M(0,z)$.
That is, $u(x,t) \equiv 0$.
\end{example}

\vspace*{0.5cm}
\setcounter{equation}{0}
\section{Non--existence of the global solutions in the quarter--plane} \label{non}
It proves that for wide classes of the initial conditions $f(x)$ the global
solutions satisfying conditions of  Proposition \ref{prtca} do not exist. Using 
small energy asymptotics of the corresponding Weyl--functions
we explicitly construct in this section such a class of initial conditions.

First, we describe the explicit construction of the potentials
and Weyl functions from Theorem 0.1 and Proposition 2.2 in \cite{GKS3}. 
For this purpose
we fix an integer $n>0$ and three matrices, namely, an $n \times n$
matrix $\a$ and $n \times m$ matrices $\vt_k$, $k=1,2$, such that
\begin{equation} \label{s2}
\alpha - \alpha^{*}= \vartheta_{1}\vartheta_{2}^{*}-\vartheta_{2}
\vartheta_{1}^{*}.
\end{equation}
The triple $\{\a, \, \vt_1, \, \vt_2\}$, which satisfies \eqref{s2}, is called
{\it admissible}.
Consider Sturm--Liouville system \eqref{r1}
where $u$ is determined by the triple $\{\a, \, \vt_1, \, \vt_2\}$. Namely, put
\begin{eqnarray} \label{s3}  
u(x)&=&2 \{ ( \Lambda_{2}(x)^{*}S(x)^{-1} \Lambda_{2}(x))^{2}
+ \Lambda_{1}(x)^{*}S(x)^{-1} \Lambda_{2}(x)
\\
&&
+\Lambda_{2}(x)^{*}S(x)^{-1}
\Lambda_{1}(x) \},
\end{eqnarray}
where
\begin{equation} \label{s3'}   \Lambda (x)=
\left [ \begin{array}{c} \Lambda_{1}(x)
\\
\Lambda_{2}(x) \end{array} \right ]=e^{x \beta }
\left [ \begin{array}{c} \vartheta_{1}
\\
\vartheta_{2} \end{array} \right ],
\hspace{1em}
\beta =
\left [ \begin{array}{lr}
0 & \alpha
\\
-I_{n} & 0
\end{array} \right ],
\end{equation}
\begin{equation} \label{s4}
S(x)=I_n+ \int_{0}^{x} \Lambda _{2}(y) \Lambda _{2}(y)^{*}dy,
\hspace{1em}x \geq 0 .
\end{equation}

\vspace*{0.25cm}

\begin{theorem}\label{EW}\cite{GKS3} Let $u$ be determined by the
admissible triple  $\{\a, \, \vt_1, \, \vt_2\}$ via formulas \eqref{s3}--\eqref{s4}
and let $Y$ satisfy \eqref{r1} and \eqref{r5}. Then,
for any sufficiently large values of $\Im{\sqrt{z}}$ $(z,\, \sqrt{z} \, \in \BC_+)$
we have
\begin{equation} \label{r6}
\int_0^\infty
\left[ \begin{array}{lr}
i \phi (\sqrt{z})^* & I_m  \end{array} \right]
Y(x,z)^*Y(x,z)
\left[ \begin{array}{c}
 -i \phi (\sqrt{z})\\ I_m \end{array} \right]
dx            < \infty, 
\end{equation}
where
\begin{equation}\label{s6}
\phi ( \sqrt{z})=\Big(\vp_{2}(z)+
\frac{2i}{\sqrt{z}}I_{m}\Big)\vp_{1}(z),
\end{equation}
and the matrix functions $\vp_{1}$ and $\vp_{2}$ are rational matrix functions
  given  by the realizations:
\begin{equation}\label{s7}
\vp_{1}(z)^{-1}=I_{m}+B^{*} \hat{J}(zI_{2n+m}-A)^{-1}B,
\hspace{1em}
\vp_{2}(z)=-I_{m}+C(zI_{2n+m}-A)^{-1}B,
\end{equation}
\begin{eqnarray} \label{s9}
&&A= \left[ \begin{array}{lcr} \alpha^{*} & 0 & 0 \\
\vartheta_{1}^{*} &  0 & 0 \\
\vartheta_{2}  \vartheta_{2}^{*} &  \vartheta_{1} &  \alpha \end{array} \right]
, \hspace{1em}B= \left[ \begin{array}{c}
  \vartheta_{1}+  \vartheta_{2}(I_{m}+  \vartheta_{2}^{*}  \vartheta_{2}
)
\\
I_{m}+  \vartheta_{2}^{*}  \vartheta_{2}
\\
- \vartheta_{2}
\end{array} \right]
, \\ \nonumber &&
\hat{J}= \left[ \begin{array}{lcr} 0 & 0 & I_{n} \\
0 &  I_{m} & 0 \\ I_{n} &  0 & 0 \end{array} \right]
, \quad
C= \left[ \begin{array}{lcr}
\vartheta_{2}^{*} & I_{m}-  \vartheta_{2}^{*}  \vartheta_{2}
& - \vartheta_{1}^{*}+( I_{m}-  \vartheta_{2}^{*}  \vartheta_{2})
\vartheta_{2}^{*}
\end{array} \right] .
\end{eqnarray}
Moreover, for any sufficiently large values of $\Im{\sqrt{z}}$ the matrix $\phi (\sqrt{z})$, such that
\eqref{r6} holds, is unique.
\end{theorem}
According to Theorem \ref{EW} inequality \eqref{r4} holds for
\begin{equation} \label{s10}
M(z)=-\phi (\sqrt{z})^{-1}
\end{equation}
and sufficiently large values of $\Im{\sqrt{z}}$.  From inequality (2.61) in \cite{GKS3} follows also
that matrices $M(z)$ satisfying \eqref{r4} for sufficiently large values of $\Im{\sqrt{z}}$ are unique.

By Proposition  2.3 in \cite{GKS3} we have
\begin{equation} \label{s11}
\sup_{0 \leq x<\infty} \| u(x) \|<\infty.
\end{equation}
Thus, there is a unique Weyl function $M(z)$ of system
\eqref{r1} and this Weyl function satisfies \eqref{r4}.
Therefore, equality \eqref{s10} defines the  Weyl function $M(z)$
for sufficiently large values of $\Im{\sqrt{z}}$. Note also that the Weyl function $M$ and the matrix function
$\phi (\sqrt{z})^{-1}$
are meromorphic.

\vspace*{0.25cm}

\begin{corollary}\label{CyWF}
The Weyl function $M$ of system \eqref{r1}, where $u$ has the form
 \eqref{s3}, is given by formulas  \eqref{s6}--\eqref{s10} for all $z \in \BC_+$ excluding a finite number of points.
\end{corollary}
Relation \eqref{s3} can be rewritten as
\begin{equation} \label{s12}
u=2( \Om_{22}^2+ \Om_{12}+ \Om_{21}); \quad
\Om_{kj}:=\Lam_k^*S^{-1}\Lam_j, \quad k,j=1,2.
\end{equation}
The derivatives of  $u$ are calculated in \cite{GKS3} using \eqref{s3}--\eqref{s4}.
In particular,  from the expressions (5.16) and (5.17) in \cite{GKS3} for the derivatives of
$\Om_{kj}$ one can get
\begin{equation} \label{s12'}
u_x=2( \Lam_2^*\a^*S^{-1}\Lam_2+\Lam_2^*S^{-1}\a \Lam_2-2\Om_{11}
- \Om_{12}\Om_{22}- \Om_{22}\Om_{21})-
u\Om_{22}-\Om_{22}u.
\end{equation}
Formula (5.31) in \cite{GKS3} has the form
 \begin{align}\label{s13}
 3u^{2}- \frac{ \partial^{2} u }{ \partial x^{2}} =&
8(\Om_{21}\Om_{12}+
 \Lambda_{2}^{*}S^{-1} \alpha \Lambda_{2} \Om_{22}+\Om_{22} \Lambda_{2}^{*}
  \alpha^{*}S^{-1} \Lambda_{2} \\& \nonumber+
\Lambda_{2}^{*}S^{-1} \alpha \Lambda_{1}+
\Lambda_{1}^{*}
  \alpha^{*}S^{-1} \Lambda_{2}).
\end{align}
Formula (5.37) in \cite{GKS3} after some cancellations takes the form
\begin{align}
& \nonumber \frac{ \partial  }{ \partial x}(3u^{2}-
\frac{ \partial^{2} u }{ \partial x^{2}})=
8 \{
\Lambda_{2}^{*}
(  \alpha^{*})^{2}S^{-1} \Lambda_{2}-
\Lambda_{1}^{*}
 \alpha^{*}S^{-1} \Lambda_{1}-\Lambda_{1}^{*}
 \alpha^{*}S^{-1} \Lambda_{2}\Om_{22} \\
& \nonumber 
 +
\Lambda_{2}^{*}S^{-1} \alpha ^{2} \Lambda_{2} -
\Lambda_{1}^{*}
  S^{-1} \alpha \Lambda_{1}-\Om_{22} \Lambda_{2}^{*}
  S^{-1} \alpha \Lambda_{1}
- \Lambda_{2}^{*}S^{-1} \alpha  \Lambda_{2}
( \Om_{22}^{2}+\Om_{21} )\\& \nonumber
-( \Lambda_{2}^{*}S^{-1} \alpha  \Lambda_{1}
+ \Lambda_{1}^{*}S^{-1} \alpha  \Lambda_{2}
+\Om_{22} \Lambda_{2}^{*}S^{-1} \alpha  \Lambda_{2})\Om_{22}-
( \Om_{22}^{2}+\Om_{12})
\Lambda_{2}^{*}
 \alpha^{*}S^{-1} \Lambda_{2}\\& \nonumber
 - \Om_{22} ( \Lambda_{2}^{*}
 \alpha^{*}S^{-1} \Lambda_{1}+ \Lambda_{1}^{*}
 \alpha^{*}S^{-1} \Lambda_{2}+
\Lambda_{2}^{*}
 \alpha^{*}S^{-1} \Lambda_{2} \Om_{22})
-(\Om_{22}\Om_{21}+\Om_{11})\Om_{12}
\\&\label{s14}
-\Om_{21}(
\Om_{12}\Om_{22}+\Om_{11}) \}.
\end{align}
Our next proposition follows from \eqref{s12}--\eqref{s14}.

\vspace*{0.25cm}

\begin{proposition} \label{Pn43}
Let 
\begin{equation} \label{s15}
\a=\a^*, \quad \vt_1^*\a\vt_1=0, \quad \vt_2=0.
\end{equation}
Then the triple
$\{\a, \, \vt_1, \, \vt_2\}$ is admissible and
\begin{equation} \label{s16}
u(0)=u_{xx}(0)=u_{xxx}(0)=0, \quad u_x(0)=-4\vt_1^*\vt_1.
\end{equation}
\end{proposition}
\begin{proof}.
As $\a=\a^*$ and $\vt_2=0$ the identity \eqref{s2} holds, that is,
the triple
$\{\a, \, \vt_1, \, \vt_2\}$ is admissible. According to \eqref{s3'} and \eqref{s15}
we have $\Lam_2(0)=\vt_2=0$.
As $\Lam_2(0)=0$ we have also $\Om_{21}(0)=
\Om_{12}(0)=\Om_{22}(0)=0$,
 and so formula \eqref{s12} implies $u(0)=0$. 
 Taking into account that
\begin{equation} \label{s17}
u(0)=0, \quad \Lam_2(0)=0, \quad \Om_{21}(0)=
\Om_{12}(0)=\Om_{22}(0)=0, \quad S(0)=I_n,
\end{equation}
we derive from  formulas  \eqref{s3'}  and  \eqref{s12'}   the equality
$u_x(0)=-4\vt_1^*\vt_1$.
Moreover, formulas  \eqref{s13} and  \eqref{s17} yield $u_{xx}(0)=0$.
By  \eqref{s14} and  \eqref{s17} we have 
\begin{equation} \label{s18}
u_{xxx}(0)=8\Lam_1(0)^*\Big(\a^*S(0)^{-1}+S(0)^{-1}\a\Big)\Lam_1(0)=8\vt_1^*(\a+\a^*)\vt_1.
\end{equation}
Finally, in view of  \eqref{s15} and \eqref{s18} we get $u_{xxx}(0)=0$.
\end{proof}
 
The first three equalities in \eqref{s16} mean that the initial condition $u(x,0)=u(x)$ for KdV  
complies with the boundary conditions $u(0,t)=u_{xx}(0,t)=0$.

\vspace*{0.25cm}

\begin{example} Consider the case
\begin{equation} \label{s19}
\a=0, \quad \vt_2=0.
\end{equation}
It is immediate that  \eqref{s15} holds, that is, the conditions
of Proposition \ref{Pn43} are fulfilled. It easily follows from
\eqref{s3'}, \eqref{s4}, and \eqref{s19} that
\begin{equation} \label{s20}
e^{x\b}=I_{2n}+x\b, \quad \Lam_1\equiv \vt_1, \quad \Lam_2(x)=-x \vt_1, \quad
S(x)=I_n+\frac{1}{3}x^3\vt_1\vt_1^*.
\end{equation}
Taking into account \eqref{s20}, we derive from \eqref{s3} that
\begin{equation} \label{s21}
u(x)=2x^4\vt_1^*\big(I_n+\frac{1}{3}x^3c\big)^{-1}c\big(I_n
+\frac{1}{3}x^3c\big)^{-1}\vt_1-4x\vt_1^*\big(I_n+\frac{1}{3}x^3c\big)^{-1}\vt_1, 
\end{equation}
where $c:=\vt_1\vt_1^*$. The Weyl function of system  \eqref{r1}, where $u$ is given
by \eqref{s21}, is constructed using \eqref{s6}--\eqref{s10} and \eqref{s19}. First note that
\begin{equation} \label{s22}
(zI_{2n+m}-A)^{-1}=\left[ \begin{array}{lcr} z^{-1}I_n & 0 & 0 \\
z^{-2}\vartheta_{1}^{*} &  z^{-1}I_m & 0 \\
z^{-3}c& z^{-2} \vartheta_{1} &  z^{-1}I_n \end{array} \right].
\end{equation}
Hence, we obtain
\begin{align} \label{s23}
&\vp_1(z)^{-1}=I_m+z^{-1}I_m+2z^{-2}\wh c+z^{-3}\wh c^2, \quad \wh c:=\vt_1^*\vt_1, \\
& \label{s24} \vp_2(z)=-I_m+z^{-1}I_m-z^{-3}\wh c^2.
\end{align}
Substitute \eqref{s23} and \eqref{s24} into \eqref{s6},  and substitute the result  into \eqref{s10}
to get  
\begin{equation} \label{s25}
M(z)=\sqrt{z}\Big(z^3 I_m +\big(zI_m+\wh c\big)^2\Big)
\Big(\sqrt{z}\big(z^3 I_m -z^2I_m+\wh c^2\big)-2iz^3I_m\Big)^{-1}.
\end{equation}
We have $\wh c=\vt_1^*\vt_1\geq 0$. Assume for simplicity
$\wh c>0$. Then, according to \eqref{s25} the low energy asymptotics
of $M$ is given by the formula
\begin{equation} \label{s26}
M(z)=I_m+2z\wh c^{-1}+O(z^2) \quad (z \to 0).
\end{equation}
\end{example}
Though equalities \eqref{s16} for $u(x,0)=u(x)$  comply with the boundary
conditions $u(0,t)=u_{xx}(0,t)=0$ the following non--existence proposition is true. 

\vspace*{0.25cm}

\begin{proposition}\label{glne1} 
There is no solution $u$ of the KdV equation
with a negative dispersion term in the quarter--plane $x\geq 0, t\geq 0$,
such that $u(x,t)$ satisfies conditions of Proposition \ref{prtca}, where the initial
condition  in \eqref{2.11} is determined by the admissible
triple $\{0,\,\vt_1, \,0\}$  $\,(\wh c=\vt_1^*\vt_1>0)$, namely, $u(x,0)$   has the form:
\begin{equation} \label{s27}
u(x,0)=2x^4\vt_1^*\big(I_n+\frac{1}{3}x^3c\big)^{-1}
c\big(I_n+\frac{1}{3}x^3c\big)^{-1}\vt_1-4x\vt_1^*\big(I_n+\frac{1}{3}x^3c\big)^{-1}\vt_1.
\end{equation}
\end{proposition}
\begin{proof}.  
 We prove this proposition by contradiction. Suppose that $u(x,t)$ described
in the proposition exists.  Then $M(0,z)=M(z)$, where $M$ is given by \eqref{s25}.
Hence, by Proposition \ref{prtca} the Weyl function of system \eqref{2.19} is given
by the formula 
\begin{equation} \label{s28}
M_D(-4 z^{\frac{3}{2}})=\frac{1}{\sqrt{z}}(I_m+M(z))(I_m-M(z))^{-1}
\end{equation}
for sufficiently large values of  $\Im \sqrt{z}$, where $z$ belongs to the sector  $\frac{2}{3}\pi<\arg (z) <\pi$.
Recall that as a Wel function $M_D(\zeta)$  is a Herglotz function  ($\zeta \in \BC_+$)
and that $M(z)$ is meromorphic in $\BC$.  Note that
\begin{equation} \label{s29}
\frac{2}{3}\pi<\arg (z) <\frac{4}{3} \pi
\end{equation}
implies $-4 z^{\frac{3}{2}} \in \BC_+$. Therefore \eqref{s28} holds in the sector \eqref{s29}.
The asymptotics  \eqref{s26} holds in $\BC$ and, in particular, in the sector \eqref{s29} too.
Moreover, according to \eqref{s26} and \eqref{s28} the low energy asymptotics of $M_D$ 
has the form
\begin{equation} \label{s30}
M_D(-4 z^{\frac{3}{2}})=-z^{-\frac{3}{2}}\big(I_m+O(z)\big)\wh c, \quad z \to 0,
\end{equation}
which contradicts the Herglotz property of $M_D$.
\end{proof}
Put
\begin{align} \label{b1}
& \Lambda (x,t)=
\left [ \begin{array}{c} \Lambda_{1}(x,t)
\\
\Lambda_{2}(x,t) \end{array} \right ]=e^{x \beta +4t \beta^{3} }
\left [ \begin{array}{c} \vartheta_{1}
\\
\vartheta_{2} \end{array} \right ],
\hspace{1em}
\beta =
\left [ \begin{array}{lr}
0 & \alpha
\\
-I_{n} & 0
\end{array} \right ], \\
&  \label{b2}
S(x,t)=I_n+P_{1}[0 \hspace{1em}e^{x \beta +4t \beta^{3} }]e^{x \omega
+4t \omega^{3} }
\left [ \begin{array}{c} P_{1}^{*}
\\
0 \end{array} \right ],
\hspace{1em} \omega = \left [ \begin{array}{lr} \beta^{*} & 0
\\
b & - \beta \end{array} \right ],
\end{align}
where $\{\a, \, \vt_1, \, \vt_2\}$ is an
admissible triple,
$P_{1}=[0 \hspace{1em}I_{n}]\,$,
$b=
\left [ \begin{array}{c} \vartheta_{1}
\\
\vartheta_{2} \end{array} \right ]
[  \vartheta_{1}^{*} \hspace{1em} \vartheta_{2}^{*}]$.
Then, according to Theorem 0.5 in \cite{GKS3} the matrix function $u(x,t)$,
given by \eqref{s12}  in the points of invertibility of $S$, satisfies KdV \eqref{1.1}.
Notice that   $\Lam(x,0)$ and $S(x,0)$ defined above coincide with
$\Lam(x)$ and $S(x)$ in \eqref{s3'} and \eqref{s4}, respectively.
Moreover, according to Chapter 5 in \cite{GKS3}  equalities
\eqref{s12'}--\eqref{s14}  hold for each $t$.
Finally, from (5.6) and (5.9) in \cite{GKS3} we have
\begin{equation} \label{b3}
S_x=\Lam_2\Lam_2^*, \quad S_t=-4\big(\a \Lam_2\Lam_2^*+ \Lam_2\Lam_2^*\a^*
+\Lam_1\Lam_1^*\big).
\end{equation}
(We changed $\Lam(x,t)$ into $\Lam(x,-t)$, $S(x,t)$ into $S(x,-t)$, and $u(x,t)$
into $u(x,-t)$ in the expressions in \cite{GKS3} to obtain  KdV solutions with a
negative dispersion term.) 

\vspace*{0.25cm}

\begin{example}\label{BU} Blow-up solutions. \\
Consider again the case \eqref{s19} of the triple $\{0,\,\vt_1,\,0\}$, where $\vt_1\not=0$.  By \eqref{s19}   we see that
$\b^2=\b^3=0$. As $\b^3=0$ formulas \eqref{s19}  and   \eqref{b1} imply
\begin{equation} \label{b4}
\Lam_1(x,t)\equiv \vt_1, \quad \Lam_2(x,t)=-x\vt_1
\end{equation}
(compare with \eqref{s20}). In particular, we get
$\Lam_2(0,t)\equiv 0$. Hence, in view of \eqref{s12} and \eqref{s13}
we derive
\begin{equation} \label{b5}
u(0,t)=u_{xx}(0,t)=0 \end{equation}
in the points of invertibility of $S(0,t)$. 
It follows from  \eqref{b3},  \eqref{b4}, and equality $\a=0$ that
\begin{equation} \label{b6}
S(x,t)=I_n+\big(\frac{1}{3}x^3-4t\big)c, \quad c=\vt_1\vt_1^*.
\end{equation}
Substitute \eqref{b4} and \eqref{b6} into \eqref{s12} to get
\begin{align}\nonumber
u(x,t)=&2x^4\vt_1^*\Big(I_n+\big(\frac{1}{3}x^3-4t\big)c\Big)^{-1}c
\Big(I_n+\big(\frac{1}{3}x^3-4t\big)c\Big)^{-1}
\vt_1
\\ &\label{b7}
-4x\vt_1^*\Big(I_n+\big(\frac{1}{3}x^3-4t\big)c\Big)^{-1}\vt_1.
\end{align}
The blow-up should occur when $\det S(x,t)$ turns
to zero.  In the simplest case $n=1$ formula \eqref{b7} takes the form
\begin{equation} \label{b8}
u(x,t)=\frac{\frac{2}{3}cx^4+16ct-4x}{\Big( 1+\frac{1}{3}cx^3-4ct\Big)^2}\vt_1^*\vt_1,
\end{equation}
and for $t \geq \frac{1}{4c}$ we have singularity at $x=\big(3(4ct-1)/c\big)^{\frac{1}{3}}$.
\end{example}
Our next proposition deals with the case, where $\det\a \not= 0$
and low energy asymptotics of $M$ is different from the asymptotics in \eqref{s26}
but the global solutions $u$ again do not exist.

\vspace*{0.25cm}

\begin{proposition}\label{glne2} 
There is no solution $u$ of the KdV equation
with a negative dispersion term in the quarter--plane $x\geq 0, t\geq 0$,
such that $u(x,t)$ satisfies conditions of Proposition \ref{prtca}, where 
$u(x,0)$  is determined by the 
triple $\{\a,\,\vt_1, \,0\}$, which satisfies relations 
\begin{equation} \label{s31}
\a=\a^*,  \quad \vt_1^*\a\vt_1=0, \quad \det \a\not=0, \quad \det(I_m\pm \vt_1^*\a^{-1}\vt_1)\not=0,
\quad \vt_1^*\a^{-1}\vt_1\not\leq 0.
\end{equation}
\end{proposition}

\begin{proof}. 
 As $\vt_2=0$,  $\a=\a^*$,   and $\vt_1^*\a\vt_1=0$ the triple
is admissible and equalities  \eqref{s16} hold, that is, the initial condition
complies with the boundary conditions $u(0,t)=u_{xx}(0,t)=0$.
By \eqref{s9} we have
\begin{align} \label{s32}
& (zI_{2n+m}-A)^{-1}=\left[ \begin{array}{lcr} (zI_n-\a)^{-1} & 0 & 0 \\
z^{-1}\vartheta_{1}^{*}(zI_n-\a)^{-1} &  z^{-1}I_m & 0 \\
z^{-1}c_1(z)c(zI_n-\a)^{-1}& z^{-1}c_1(z) \vartheta_{1} &  c_1(z)\end{array} \right], \\
\label{s33} &B^*= \left[ \begin{array}{lcr}
\vartheta_{1}^{*} & I_{m}
&0
\end{array} \right] , \quad
C= \left[ \begin{array}{lcr}
0 & I_{m}
& - \vartheta_{1}^{*}
\end{array} \right] ,
\end{align}
where
\[
c_1(z)=(zI_n-\a)^{-1}, \quad c=\vt_1\vt_1^*.
\]
According to \eqref{s7}, \eqref{s32}, and \eqref{s33} we have
\begin{align} \label{s34}
&\vp_1(z)^{-1}=I_m+z^{-1}\big(I_m+\vt_1^*(zI_n-\a)^{-1}\vt_1\big)^2,
\\
& \label{s35} \vp_2(z)=-I_m+z^{-1}\Big(I_m-\big(\vt_1^*(zI_n-\a)^{-1}\vt_1\big)^2\Big).
\end{align}
By \eqref{s6},  \eqref{s10}, \eqref{s34}, and \eqref{s35} the low energy
asymptotics of $M(0,z)$ has the form
\begin{align}\nonumber
M(0,z)=&-\big(I_m+\vt_1^*\a^{-1}\vt_1\big)^{-1}
\Big(I_m-\vt_1^*\a^{-1}\vt_1-2i\sqrt{z}\big(I_m+\vt_1^*\a^{-1}\vt_1\big)^{-1}\Big)
\\ &
+O(z), \quad z \to 0.  \label{s36}
\end{align}
Finally, in a way similar to the corresponding part of the proof of Proposition \ref{glne1} we
assume that $u(x,t)$ satisfying conditions of Proposition \ref{prtca} exists and get
\begin{equation} \label{s37}
M_D(-4 z^{\frac{3}{2}})=\frac{1}{\sqrt{z}}\vt_1^*\a^{-1}\vt_1+O(1), \quad z \to 0
\end{equation}
in the sector \eqref{s29}. In view of  the last relation in \eqref{s31} this means that $M_D$
does not belong to Herglotz class and we come to a contradiction.
\end{proof}



\section*{Acknowledgements}
This work was supported by the Austrian Science Fund (FWF) under Grant  no. Y330. \\
The author is very grateful to F. Gesztesy for his help and interest in the topic.



\end{document}